\documentclass[a4paper,11pt]{elsarticle}
\usepackage{amsthm}
\usepackage{amsmath}
\usepackage{amssymb}
\usepackage{color}
\usepackage{enumerate}
\usepackage{verbatim}
\usepackage{subfig}
\usepackage{pgf,tikz}
\usetikzlibrary{snakes}
\usetikzlibrary{arrows}
\usepackage{times}
\usepackage{enumerate}

\makeatletter
\def\ps@pprintTitle{%
  \let\@oddhead\@empty
  \let\@evenhead\@empty
  \def\@oddfoot{\reset@font\hfil\thepage\hfil}
  \let\@evenfoot\@oddfoot
}
\makeatother

\newcommand\junk[1]{}

\newtheorem{theorem}{Theorem}

\newtheorem{corollary}[theorem]{Corollary}
\newtheorem{definition}[theorem]{Definition}
\newtheorem{remark}[theorem]{Remark}

\newtheorem{lemma}[theorem]{Lemma}

\def\bd{\mathbf{d}}
\def\ds{\mathbf{s}}

\def\J{\mathcal{J}}
\def\cA{\mathcal{A}}
\def\cB{\mathcal{B}}

\begin{document}

\begin{frontmatter}
\title{On Realizations of a Joint Degree Matrix \tnoteref{t1}}
\author[sc,scimi]{\'{E}va Czabarka}
\author[scimi]{Aaron Dutle
}
\author[renyi]{P\'eter L. Erd\H os\fnref{elp}}
\author[renyi]{Istv\'an Mikl\'os\fnref{mik}}
\address[sc]{Department of Mathematics, University of South Carolina, \\1523 Greene St., Columbia, SC, 29208, USA \\ {\tt email:} czabarka@math.sc.edu}
\address[scimi]{Interdisciplinary Mathematics Institute, University of South Carolina, \\1523 Greene St., Columbia, SC, 29208, USA \\ {\tt email:}  dutle@mailbox.sc edu}
    
\address[renyi]{Alfr\'ed R{\'e}nyi Institute, Re\'altanoda u 13-15 Budapest, 1053 Hungary\\
        {\tt email:} $<$erdos.peter,miklos.istvan$>$@renyi.mta.hu}

\fntext[elp]{PLE   was supported in part by the Hungarian NSF, under contract NK 78439.}
\fntext[mik]{IM was supported in part by a Bolyai postdoctoral stipend and by the Hungarian NSF, under contract F61730.}

\tnotetext[t1]{The authors acknowledge financial support from grant \#FA9550-12-1-0405 from the U.S. Air Force Office of Scientific Research (AFOSR) and the Defense Advanced Research Projects Agency (DARPA).}
\begin{abstract}
The joint degree matrix of a graph gives the number of edges between vertices of degree $i$ and 
degree $j$ for every pair $(i,j)$. One can perform restricted swap operations to transform 
a graph into another with the same joint degree matrix. We prove that the 
space of all realizations of a given joint degree matrix over a \mbox{fi}xed vertex set is connected via 
these restricted swap operations. This was claimed before, but there is an error in the 
previous proof, which we illustrate by example.  We also give a simplified proof of the necessary 
and sufficient conditions for 
a matrix to be a joint degree matrix. Finally, we address some of the issues concerning the 
mixing time of the corresponding MCMC method to sample uniformly from these realizations.     
\end{abstract}
\begin{keyword} degree sequence\sep joint degree matrix\sep  Havel-Hakimi algorithm \sep restricted swap\sep Erd\H{o}s-Gallai theorem.
\MSC[2010] 05C82\sep 90B10\sep 90C40 \end{keyword}
\end{frontmatter}

\section{Introduction}
\noindent In  recent years there has been a large (and growing) interest in real-life social 
and biological networks. One important distinction between these two network types lies in their
overall structure: the first type typically have a few very high degree vertices and many low 
degree vertices with high {\em assortativity} (where a vertex is likely to be adjacent to vertices of 
similar degree), while the second kind is generally {\em disassortative} (in which low degree 
vertices tend to attach to those of high degree). It is well known, the {\em degree sequence}  
alone cannot capture these differences. There are several approaches to address this problem. 
See the paper of Stanton and Pinar (\cite{stanton}) for a detailed description of the current 
state-of-the-art.

\medskip\noindent
In this paper, we address the  {\em joint degree distribution} (or {\em JDD}) model. This model 
is more restrictive than the degree distribution, but it provides a way to enhance
results based on degree distribution.  In essence, the degree distribution of a graph can be considered  as the probability that a vertex selected uniformly at random will be of degree $k$. Analogously,  the joint degree distribution describes the probability that a randomly selected edge 
of the graph connects vertices of degree $k$ and $\ell$.

\medskip\noindent
Amanatidis, Green and Mihail \cite{agm08}  and Stanton and Pinar  \cite{stanton} introduced the {\em joint degree matrix} (or JDM for short) model which is a version of JDD. In essence, the JDD gives (for each $i$ and $j$) the {\em probability} that an edge of the graph 
connects a vertex of degree $i$ to a vertex of degree $j$, while JDM tells us  the exact {\em number} of edges 
between vertices of degrees $i$ and $j$. We will give precise definitions in Section \ref{sec:def}.

\medskip\noindent
In a still unpublished paper  \cite{agm08},  an Erd\H{o}s-Gallai type  theorem was presented 
for joint degree matrices. The lecture \cite{schmitt} sketched its original proof. Stanton and Pinar
\cite{stanton} gave a new, constructive proof for this theorem. In Section \ref{sec:EG}, we present a simpler proof which gives a more general construction algorithm. 

\medskip\noindent
Also in \cite{stanton}, Stanton and Pinar proposed a restricted version of the classical {\em swap operation} (in their words: {\em rewiring}) to transform one realization of a JDM into another one. 
They describe this operation in terms of a generalized configuration model (for the original model 
see \cite{bollobas}), in which a swap is essentially a manipulation of perfect matchings in a bipartite graph. Indeed, if one also considers realizations that are {\em multigraphs} (i.e., graphs allowing loops and multiple edges), their generalized configuration model describes all possible realizations. 
Using a theorem of Ryser (\cite{ryser}) on this generalized configuration model, 
Stanton and Pinar proved that the space of all multigraph realizations is connected. 
They address the connectivity of the space of all (simple) graph realizations of a JDM 
(those without multiple edges or loops), and claim to prove that restricted swap operations make the space of these realizations connected. We show in Section \ref{sec:irred} that their proof is flawed, and present a correct proof of this result in Section \ref{sec:homo}. 

\medskip\noindent
Stanton and Pinar also concluded \cite{stanton} that the corresponding MCMC algorithms that sample multigraph realizations and simple realizations of a JDM are both fast mixing. 
They claimed to give a proof for the first statement, and supported the second statement 
with experimental results. We address both of these claims in Section  \ref{sec:MCMC}.

\medskip\noindent
Finally, in Section~\ref{sec:further} we discuss some open questions.

\section{Definitions}\label{sec:def}
\noindent For the remainder of the paper, unless otherwise noted, all graphs (and by extension all realizations of a JDM) are simple graphs without isolated vertices, and the vertices are labeled. 
Let $G=(V,E)$ be an $n$-vertex graph with {\em degree sequence} $\bd(G)=(d(v_1),\ldots, d(v_n))$. We denote the maximum degree by $\Delta$, and for $1\le i\le\Delta$, 
the set of all vertices of degree $i$ is $V_i$.
The {\em degree spectrum} $\ds_G(v)$ is a vector with $\Delta$ components,
where $\ds_G(v)_i$ gives the number of vertices of degree $i$ adjacent to $v$ 
in the graph $G$. While in {\em graphical realizations} of a degree sequence $\bd$ 
the degree of any particular vertex $v$ is prescribed, its degree spectrum may vary.

\begin{definition}\label{def}
The {\em joint degree matrix} $\J(G)=[\J_{ij}]$ of the graph $G$ is a $\Delta\times\Delta$ matrix where
$\J_{ij}=|\{xy\in E(G):\,x\in V_i,y\in V_j\}|$.
If, for a $k \times k$ matrix $M$ there exists a graph $G$ such that 
$\J(G)=M$, then $M$ is called a {\em graphical JDM}.
\end{definition}
\begin{remark}\label{th:size}
The degree sequence of the graph is determined by its JDM:
\begin{equation}\label{eq:size}
\qquad |V_i| = \frac{1}{i} \left (\J_{ii} + \sum_{\ell=1}^{\Delta} \J_{i \ell} \right ). \qquad \qed
\end{equation}
\end{remark}
\noindent Let $G=(V,E)$ be a  graph and $a,b,c,d$ be distinct vertices where $ac, bd \in E$ while $bc, ad \not \in E$. If $G$ is bipartite, we also require that $a,b$ are in the same class of the bipartition. 
Then $G'=(V,E')$ with 
\begin{equation}\label{eq:swap}
E'= \left(E \setminus \{ac, bd\}\right) \cup \{bc, ad\}
\end{equation}
is another realization of the same degree sequence (and if $G$ is bipartite then $G'$ remains 
bipartite with the same bipartition). The operation in (\ref{eq:swap}) is called a {\em swap}, 
and we denote it by $ac, bd \Rightarrow bc, ad$.  Swaps are used in the 
Havel-Hakimi algorithm (\cite{havel} and \cite{hakimi}).  Petersen \cite{pet} was the 
first to prove that any realization of a degree sequence can be transformed into any 
other realization using only swaps. The corresponding result for bipartite graphs 
was proved by Ryser \cite{ryser}.  

\medskip\noindent An arbitrarily chosen swap operation on $G$ may alter the JDM, so we introduce the {\em restricted swap operation} (or for brevity {\em RSO}), which preserves the JDM.
\begin{definition}
A swap operation is a {\em RSO} if it is a swap operation of the form $ac, bd \Rightarrow bc, ad$, with the additional restriction that there is an $i$ such that $a,b \in V_i$.
\end{definition}
\noindent
It is clear that RSOs indeed keep the JDM unchanged. Even more, an RSO changes only the degree spectrum of vertices $a$ and $b$, a fact that we use repeatedly.   
When we refer to swaps on graphs and bipartite graphs that are not necessarily RSOs, we use the terms {\em ordinary swaps} 
and {\em bipartite swaps}.

\section{The space of all graphical realizations---the challenges}\label{sec:irred}

\noindent Stanton and Pinar \cite{stanton} propose an inductive proof to show that the restricted swap 
operations make the space of all realizations of a JDM connected. 
They take two realizations, $G$ and $H$, of the same 
JDM, choose a vertex $v$, and using RSOs, they transform $G$ and $H$ into $G'$ and $H'$ 
with the property that 
the neighborhoods of $v$ in $G'$ and $H'$ are the same set of vertices. They 
state that after removing $v$ from $G'$ and $H'$ the JDM of the resulting graphs still agree, i.e.
$\J(G'-v)=\J(H'-v)$. Unfortunately, as the following example show,
this is not the case, not even if we
require in addition that the vertices in the neighborhood of 
$V$ have the same degree spectra in $G$ and $H$. 

\medskip\noindent
Let $G$ be a six-cycle with vertices labeled cyclically by integers $\{1,\ldots,6\}$ Let $H$ be the disjoint union of two three-cycles labeled cyclically by $\{1,2,3\}$ and $\{4,5,6\}$ respectively. 
$G$ and $H$ are $2$-regular graphs on $6$ vertices,
both have the same JDM, $\J = \left[ \begin{smallmatrix} 0&0\\0&6 \end{smallmatrix}\right]$, and the degree spectrum of any vertex $v$ is
is $s_G(v)=s_H(v)=(0,2)$.   If we consider vertex $2$, we note that it's neighbor set is $\{1,3\}$ in each graph. On the other hand, the truncated realizations given by deleting vertex $2$ from the graphs each have a different JDM. The truncated JDM for $G$ is $\left[\begin{smallmatrix} 0& 2\\ 2& 2 \end{smallmatrix}\right]$, while for $H$ the truncated JDM is $\left[\begin{smallmatrix} 1& 0\\ 0& 3 \end{smallmatrix}\right]$.

\begin{figure}[h!]
\begin{center}
\begin{tikzpicture}[line cap=round,line join=round,>=triangle 45,x=1.0cm,y=1.0cm]
\clip(-4.25,-0.5) rectangle (4.4,6.35);
\draw [line width=1.2pt] (-3.96,5.12)-- (-2.9,5.88);
\draw [line width=1.2pt] (-1.86,5.16)-- (-2.9,5.88);
\draw [line width=1.2pt] (-1.84,4.04)-- (-1.86,5.16);
\draw [line width=1.2pt] (-2.86,3.34)-- (-1.84,4.04);
\draw [line width=1.2pt] (-2.86,3.34)-- (-3.96,4.04);
\draw [line width=1.2pt] (-3.96,4.04)-- (-3.96,5.12);
\draw [line width=1.2pt] (-3.38,1.08)-- (-3.4,2.2);
\draw [line width=1.2pt] (-3.4,2.2)-- (-2.38,1.72);
\draw [line width=1.2pt] (-2.38,1.72)-- (-3.38,1.08);
\draw [line width=1.2pt] (-2.24,0.98)-- (-3.26,0.38);
\draw [line width=1.2pt] (-3.26,0.38)-- (-2.26,-0.24);
\draw [line width=1.2pt] (-2.24,0.98)-- (-2.26,-0.24);
\draw [->] (-0.62,4.64) -- (0.78,4.62);
\draw [->] (-0.7,1.04) -- (0.72,1.04);
\draw [line width=1.2pt] (2.18,5.24)-- (2.2,4.1);
\draw [line width=1.2pt] (2.2,4.1)-- (3.2,3.48);
\draw [line width=1.2pt] (3.2,3.48)-- (4.14,4.14);
\draw [line width=1.2pt] (4.14,4.14)-- (4.12,5.24);
\draw [line width=1.2pt] (2.72,1.22)-- (3.66,1.8);
\draw [line width=1.2pt] (2.8,0.5)-- (3.76,1.04);
\draw [line width=1.2pt] (3.76,1.04)-- (3.76,-0.14);
\draw [line width=1.2pt] (3.76,-0.14)-- (2.8,0.5);
\begin{scriptsize}
\fill [color=black] (-3.96,5.12) circle (2.5pt);
\draw[color=black] (-3.8,5.46) node {$1$};
\fill [color=black] (-2.9,5.88) circle (2.5pt);
\draw[color=black] (-2.76,6.22) node {$2$};
\fill [color=black] (-1.86,5.16) circle (2.5pt);
\draw[color=black] (-1.7,5.5) node {$3$};
\fill [color=black] (-1.84,4.04) circle (2.5pt);
\draw[color=black] (-1.68,4.38) node {$4$};
\fill [color=black] (-2.86,3.34) circle (2.5pt);
\draw[color=black] (-2.72,3.68) node {$5$};
\fill [color=black] (-3.96,4.04) circle (2.5pt);
\draw[color=black] (-3.82,4.38) node {$6$};
\fill [color=black] (-3.4,2.2) circle (2.5pt);
\draw[color=black] (-3.24,2.54) node {$2$};
\fill [color=black] (-3.38,1.08) circle (2.5pt);
\draw[color=black] (-3.22,1.42) node {$1$};
\fill [color=black] (-2.38,1.72) circle (2.5pt);
\draw[color=black] (-2.28,2.06) node {$3$};
\fill [color=black] (-2.24,0.98) circle (2.5pt);
\draw[color=black] (-2.12,1.32) node {$5$};
\fill [color=black] (-3.26,0.38) circle (2.5pt);
\draw[color=black] (-3.1,0.72) node {$4$};
\fill [color=black] (-2.26,-0.24) circle (2.5pt);
\draw[color=black] (-2.12,0.1) node {$6$};
\fill [color=black] (2.18,5.24) circle (2.5pt);
\draw[color=black] (2.34,5.58) node {$1$};
\fill [color=black] (4.12,5.24) circle (2.5pt);
\draw[color=black] (4.28,5.58) node {$3$};
\fill [color=black] (2.2,4.1) circle (2.5pt);
\draw[color=black] (2.34,4.44) node {$6$};
\fill [color=black] (4.14,4.14) circle (2.5pt);
\draw[color=black] (4.3,4.48) node {$4$};
\fill [color=black] (3.2,3.48) circle (2.5pt);
\draw[color=black] (3.36,3.82) node {$5$};
\fill [color=black] (2.72,1.22) circle (2.5pt);
\draw[color=black] (2.88,1.56) node {$1$};
\fill [color=black] (3.66,1.8) circle (2.5pt);
\draw[color=black] (3.84,2.14) node {$3$};
\fill [color=black] (2.8,0.5) circle (2.5pt);
\draw[color=black] (2.94,0.84) node {$4$};
\fill [color=black] (3.76,1.04) circle (2.5pt);
\draw[color=black] (4.06,1.38) node {$5$};
\fill [color=black] (3.76,-0.14) circle (2.5pt);
\draw[color=black] (4.04,0.2) node {$6$};
\end{scriptsize}
\end{tikzpicture}
\end{center}
\caption{An example of two realizations of a JDM, where vertex 2 has the same neighbor set in both realizations, but whose truncated realizations do not have the same JDM.}
\end{figure}
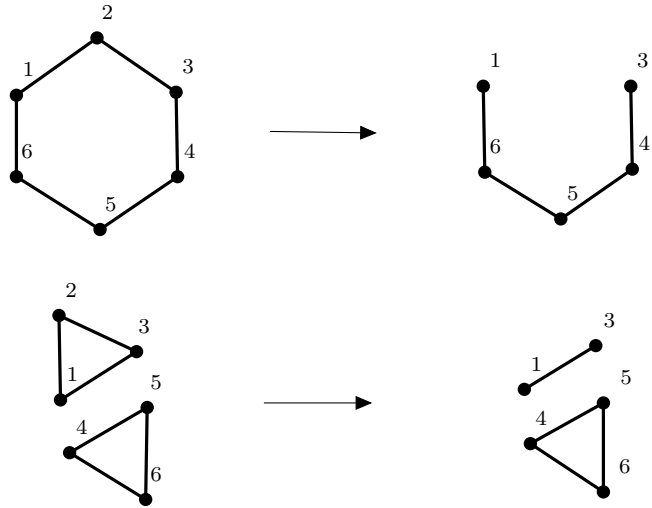

\noindent
We also note that this is not an isolated example. Many pairs of graphs can be found that demonstrate the same problem. 

\section{The space of all realizations is connected under RSOs}\label{sec:homo} 

\medskip\noindent
While Stanton and Pinar's proof is flawed, the statement of their theorem is still true, as we show here. That is, we show that  any realization of a JDM can be transformed via restricted swap operations into any other realization.  First we introduce some definitions and notation.

\medskip\noindent 
Let $\J$ be a graphical JDM.
We fix the vertex set $V$ and its partition $V_1\cup\ldots\cup V_{\Delta}$ appropriately. A {\em realization} $G$ is a graph on vertex set $V$ where the set of vertices of degree $i$ is $V_i$.

\medskip\noindent
For each $j$, set $A_j(j):=\frac{2 \J(j,j)}{|V_j|}$, and for $i\ne j$, $A_j(i):=\frac{\J(i,j)}{|V_j|}$.  Simply put, for any realization $G$ of $\J$ and for all pairs $i,j$ the quantity $A_j(i)$ gives
the average of $\ds_G(v)_i$ over all vertices $v\in V_j$.

\medskip\noindent
The set of degree $j$ vertices $V_j$ is {\em balanced} in $G$, if for each $V_i$ 
the edges connecting $V_j$ to $V_i$ are as uniformly distributed on $V_j$ as possible. In other words, for all $v\in V_j,$ and for all $i,$ we have that
$\ds_G(v)_i\in\{\lfloor A_j(i)\rfloor, \lceil A_j(i) \rceil \}$. A realization $G$ 
is {\em balanced} if $V_i$ is balanced  in $G$ for all $i$.

\medskip\noindent
We will show now that any realization of a given JDM  can be 
transformed into a balanced one via restricted swap operations. 
To this end, for a realization $G,$ a vertex $v,$ and index $i,$ we define 
$$c_G(v,i):= \lfloor | A_{d(v)}(i) - \ds_G(v)_i | \rfloor.$$ 
Clearly, $c_G(v,i)\ge 0$, and $c_G(v,i)=0$ for a $v\in V_j$ precisely
when $\ds_G(v)_i\in\{\lfloor A_j(i) \rfloor,\lceil A_j(i) \rceil\}$.
Our algorithm will be governed by the function
\begin{equation}\label{eq:distance}
C_G(j)=\sum _{v\in V_j} \sum _{i=1}^{k} c_G(v,i).
\end{equation}
When $C_G(j)=0$,  i.e. when $c_G(v,i)=0$ for all $v\in V_j$, then $V_j$ is balanced.
\begin{lemma}\label{th:homo}
If $C_G(j)\ne 0$,  then there are vertices $u, v \in V_j$ and an RSO $vw, uz \Rightarrow vz, uw$ transforming $G$ into $G'$ such that $C_{G'}(j)<C_G(j)$ and for all $\ell\ne j$ $C_{G'}(\ell)=C_{G}(\ell)$.
\end{lemma}
\begin{proof} 
Choose $u,v\in V_j$ such that $\ds_G(u)_i$ is minimal and $\ds_G(v)_i$ is maximal amongst all vertices in $V_j$.
Then 
we have $
\ds_G(u)_i\le \lfloor A_j(i) \rfloor\le\lceil A_j(i) \rceil\le \ds_G(v)_i$ with at least two strict inequalities.This implies that 
$\ds_G(u)_i < \lfloor A_j(i) \rfloor<\ds_G(v)_i$ or $\ds_G(u)_i<\lceil A_j(i)\rceil<\ds_G(v)_i$ holds.
Assume $\ds_G(u)_i < \lfloor A_j(i) \rfloor<\ds_G(v)_i$ (the other case is handled similarly).
As $u$ has fewer neighbors in $V_i$ then $v$,  there exists a $w\in V_i$ such that  $vw\in E(G)$ but $uw\not\in E(G)$.

\noindent
Since $d(v)=d(u)=j$ and $\ds_G(v)_i > \ds_G(u)_i$, there exists a $k\ne i$ such that $\ds_G(u)_k> \ds_G(v)_k$. Consequently there exists 
$ z \in V_k$ such that $uz \in E(G)$ while $vz\not\in E(G)$. Thus  
$vw, uz \Rightarrow vz, uw$  is actually an RSO. 
It is easy to see that 
$$
c_G(v,i) + c_G(u,i) -2 \le c_{G'}(v,i) + c_{G'}(u,i) \le c_G(v,i) + c_G(u,i)-1
$$
while
$$
c_G(v,k) + c_G(u,k) -2 \le c_{G'}(v,k) + c_{G'}(u,k) \le c_G(v,k) + c_G(u,k).
$$
This implies  $C_{G'}(j)<C_G(j)$. $C_{G'}(\ell)=C_G(\ell)$ for $\ell\ne j$ follows from the fact that this RSO can change only the degree spectrum of $u$ and $v$. 
\end{proof}

\noindent
This lemma easily implies
\begin{corollary}\label{th:cor}
Let $G$  be a realization of a graphical JDM. There exists a series of RSOs transforming $G$ into a balanced realization $G'$.
\end{corollary}
\begin{proof} If $G$ is balanced we are done. Otherwise let $\{j_1,\ldots,j_s\}$ be the set of indices $j$ for which $C_G(j)\ne 0$. We define
a sequence $G=G_0,G_1,\ldots,G_s$ such that for each $1\le i\le s$ there is a sequence of RSOs transforming
$G_{i-1}$ to $G_i$, $C_{G_i}(j_i)=0$ and for $\ell\ne j_i$ we have $C_{G_i}(\ell)=C_{G_{i-1}}(\ell)$. Successive applications of
Lemma~\ref{th:homo} with $j=j_i$ give that an appropriate sequence of RSOs exists for each $i$. It follows that  $G_s$ is balanced.
\end{proof}

\noindent
It remains to show that any two balanced realizations are connected via a sequence of RSOs. 
To this end, we introduce the following definitions.
We call $V_i$ {\em mixed} with respect to $V_j$ if $A_j(i)$ is not an integer. (Note that it is possible for $V_j$ to be mixed with respect to itself.)  
When $V_i$ is mixed with respect to $V_j$,  we call a vertex $v\in V_j$ {\em low for $V_i$} if $\ds_G(v)_i = \lfloor A_j(i) \rfloor,$ and call it  {\em high for $V_i$} if $\ds_G(v)_i = \lfloor A_j(i) \rfloor+1$. 

\noindent
The auxiliary bipartite graph $\cA(G,j)=(U,P;E)$, is given by $U=\{u_v:v\in V_j\}$, $P=\{p_i: V_i$ is mixed with respect to $V_j\}$, and
$E=\{u_vp_i:\,v$ is high for $V_i\}$. 

\noindent
Now we are ready to show 
\begin{lemma}\label{th:swap}
If there is a bipartite swap operation transforming 
$\cA(G,j)=(U,P;E)$ into $\cA'=(U,P;E')$, then there is an RSO transforming $G$ into $G'$ such that $\cA(G',j) = \cA'$, and  $\ds_G(v)=\ds_{G'}(v)$ for each vertex $v\notin V_j$.
\end{lemma}
\begin{proof}
Let $u_vp_i,u_wp_k\Rightarrow u_vp_k,u_wp_i$ be a bipartite swap transforming $\cA(G,j)$ into $\cA'$. Then $v,w\in V_j$, and in the graph $G$ the vertex
$v$ is high for $V_i$, $w$ is high for $V_k$,  $v$ is low for $V_k$ and $w$ is low for $V_i$.
Therefore there exists $x\in V_i$ and $y\in V_k$ such that 
$vx \in E(G)$, $wx \not \in E(G)$, $vy \notin E(G)$ and $wy\in E(G)$. It's easy to see that 
$vx,wy \Rightarrow  vy, wx$ is an RSO transforming $G$ into a $G'$ that has the desired properties.  
\end{proof}
\noindent
This easily implies  

\begin{theorem}\label{th:homo2}
If $G$ and $H$ are two balanced realizations of the same JDM, then there is a series of RSOs transforming $G$ into 
$G'$, such that $\ds_{G'}(v) =\ds_{H} (v)$ for each $v\in V$.  
\end{theorem}

\begin{proof}
We will define a sequence of graphs $G_0=G,G_1,\ldots,G_{\Delta}$ such that for 
$1\le i\le \Delta$ we have a sequence
of RSOs that transforms $G_{i-1}$ to $G_i$, with the properties that
$\ds_{G_i}(v)=\ds_H(v)$ for each $v\in V_i$ and $\ds_{G_i}(v)=\ds_{G_{i-1}}(v)$ for each $v\notin V_i$.
With $G_{i-1}$ is already defined, consider the bipartite graphs $\cA(G_{i-1},i)$ and $\cA(H,i)$. They have the same degree sequences.
Thus Ryser's  Theorem (\cite{ryser}) gives a sequence of bipartite swaps transforming one into the other. Repeated applications of
Lemma~\ref{th:swap} implies the existence
of RSOs transforming $G_{i-1}$ into a $G_i$ with the required properties.
(In fact, one can use Theorem 3.5 from \cite{distance} to determine the minimum sequence length necessary for the task.) 
The statement then follows by choosing $G'=G_{\Delta}$.
\end{proof}

\noindent
We are ready now to prove the main result.
\begin{theorem}\label{th:conn}
The space of all realizations of any given JDM is connected via RSOs.
\end{theorem}
\begin{proof}
Let $G$ and $H$ be two realizations of the same JDM. 
Corollary \ref{th:cor} will transform our realizations via RSOs into balanced realizations $G'$ and $H'$. Applying Theorem \ref{th:homo2}  transforms $G'$ via RSOs into a balanced realization $G''$ such that $\ds_{G''}(v) =\ds_{H'} (v)$ for each vertex $v$. 
For $i\ne j$ let $G_{ij}$ and $H_{ij}$ be the bipartite graphs spanned by vertex sets $V_i$ and $V_j$ in $G''$ and $H'$ respectively,
and let $G_{ii}$ and $H_{ii}$ be the corresponding graphs spanned by $V_i$. Notice that 
ordinary and bipartite swap operations in $G_{ij}$ (when $i=j$ and $i\ne j$) are in fact RSOs in 
$G''$, and the degree sequences of $G_{ij}$ and $H_{ij}$ are the same.
A straightforward application of the corresponding Havel-Hakimi algorithm (\cite{havel} and \cite{hakimi} or \cite{pet}) and Ryser's theorem (\cite{ryser}) gives us a sequence of
RSOs transforming $G''$ to $H$.

Since the inverse of any RSO is also an RSO, the proof is complete.
\end{proof}

\section{Characterization of Graphical JDMs}\label{sec:EG}
\noindent The following characterization for a square matrix $M$ with integer entries
to be a graphical JDM was proved by Amanatidis, Green and Mihail in the still unpublished paper \cite{agm08}. In the lecture of Schmitt (\cite{schmitt}) one can find a sketch of that proof. Later Stanton and Pinar gave another constructive proof.  Here we provide a
more transparent and direct approach to the construction. As it provides simple necessary and sufficient conditions for a matrix to be realized as a graphical JDM, we call the result  an {\em Erd\H{o}s-Gallai type}   theorem  (see \cite{EG}).
\begin{theorem}[Erd\H{o}s-Gallai type theorem for JDM]\label{th:EG-JDM}
A $k \times k$ matrix $\J$ is a graphical JDM if and only if the following hold. 
\begin{enumerate}[{\rm (i)}]
\item for all $i:\; $ $n_i:=\frac{1}{i} \left (\J_{ii}+\sum\limits_{j=1}^{k} \J_{ij} \right )$ \ \ is an integer;
\item for all $i:\; $ $\J_{ii}\le {n_i\choose 2}$;
\item for all $i\ne j:\; $ $\J_{ij}\le n_i n_j$.
\end{enumerate}
\end{theorem}

\begin{proof}
The necessity of the properties is trivial, so it remains to show that they are sufficient.

\medskip\noindent
Assume that $\J$ satisfies the required properties. We need to construct a graph $G$ with $\J(G)=\J$.

\medskip\noindent We fix a partition $V$ into $k$ vertex sets $W_1,\ldots,W_k$ with
$|W_i|=n_i$. For any graph $G=(V,E)$ we will use the notation $G_{ij}$ to denote the graph
on vertex set $W_i\cup W_j$ with edge set $E_{ij}=\{xy\in E: x\in W_i,y\in W_j\}$. Clearly,
$G_{ij}=G_{ji}$ and for $i\ne j$ the graph $G_{ij}$ is bipartite.
Moreover, for $\{i,j\}\ne\{i',j'\}$ the graphs $G_{ij}$ and $G_{i'j'}$ are edge-disjoint.

\medskip\noindent
We set ${\mathcal G}$ be the set of all graphs $G'$ where for each
$1\le i\le j\le k$ the graph $G'_{ij}$ has $\J_{ij}$ edges.
The conditions on $\J$ ensure that ${\mathcal G}$ is nonempty.  However, if $G'\in{\cal G}$,
then $W_i$ may not be the set of vertices of degree $i$ in $G'$, thus $G'$ is not necessarily a realization of $\J$.
On the other hand if $G'\in{\mathcal G}$ has the property that for all $i$ the set of degree $i$ vertices is $W_i$, then
$\J(G')=\J$.

\medskip\noindent
For a $G'\in{\cal G}$ let
$$
\psi(G'):=\sum_{i=1}^k\sum_{v\in W_i} |d_{G'}(v)-i|,
$$
and let $G$ be a graph minimizing $\psi$ in ${\cal G}$. Clearly, $\psi(G)\ge 0$, and
if $\psi(G)=0$, then
$W_i$ is the set of vertices of degree $i$ in $G$, and consequently $\J(G)=\J$. 

\medskip\noindent
Assume
to the contrary that $\psi(G)>0$. This means that
we have an $i$ such that $W_i$ is not the set of vertices of degree $i$ in $G$.
Since the sum of the degrees of the vertices in $W_i$ is $i|W_i|$ in $G$,
there are $x,y\in W_i$ with $d_{G}(x) < i$ and $d_{G}(y)>i$.
Thus there is a $j$ (not necessarily different from $i$) and a $z\in W_j$ s.t. $yz\in E(G_{ij})$ and $xz\not\in E(G_{ij})$. 
Let $G^*=(V,E^*)$ where $E^*=(E(G)\setminus\{yz\})\cup\{xz\}$. It is easy to see
that $G^*\in{\cal G}$ with $\psi(G^*)<\psi(G)$, a contradiction. 
\end{proof}

\noindent
We note that the proof easily translates to an algorithm to create a  realization of $\J$. In fact, as noted above, {\em every} realization of $\J$ can be generated in this manner.

\section{Some observations on the corresponding Markov chains}\label{sec:MCMC}
\medskip\noindent
The paper \cite{stanton} presents a configuration model for generating realizations of a JDM, and discusses 
Markov chains that act on this configuration model. We give a short description of the model and the Markov chains, and discuss some issues concerning the mixing times of these Markov chains. 

\subsection*{The configuration model and the Markov chains}
\noindent
A graphical JDM $\J$ determines $|V_k|$ for every $k$ in every realization by (\ref{eq:size}).
For each $k$, and each vertex $v\in V_k$, create a cloud of $k$ mini-vertices corresponding to $v$. For each edge $e$ arising from $\J_{ij}$, create two vertices labeled by $e$. One of class $i$, and one of class $j$. We connect all mini-vertices arising from $V_k$ to all edge vertices of class $k$. The result is a complete bipartite graph $K_{k|V_k|, k|V_k|}$ for each $k$.  The collection of these graphs is  called the generalized configuration model. A realization of $\J$ can be found as follows. Take a perfect matching on each bipartite graph (this is called a {\em configuration}). Define a (possibly multigraph) realization by having $v$ adjacent to $w$ if a vertex in the cloud of $v$ and a vertex in the cloud of $w$ are adjacent to the two vertices labeled by the same edge $e$.

\medskip\noindent
It is easy to see that this will generate a (perhaps multigraphical) realization of $\J$, the resulting
graph may have loops and multiple edges. \cite{stanton} proposes the following Markov chain.

\medskip\noindent
Given a starting configuration, with probability 1/2, do nothing. Otherwise, choose a random edge $v_1e_1$ of the configuration, and another random edge $v_2e_2$ in the same bipartite component as the first edge. Perform the swap $v_1e_1, v_2e_2\Rightarrow v_1e_2, v_2e_1$. 

\medskip\noindent
It is clear that this process will output another configuration, and hence a multigraphical
realization of $\J$. This Markov chain, which \cite{stanton} calls chain $\cA$, generates 
multigraphs. A secondary Markov chain, which they refer to as chain $\cB$, begins with a configuration that corresponds to a simple realization, follows the same procedure, but rejects the swap if the corresponding realization is not a simple graph. 

\subsection*{Sampling from chain $\cA$}
\noindent
Stanton and Pinar claim that chain $\cA$ allows for uniform sampling of the configuration model, and hence of multigraphs. They correctly concluded that this chain is rapidly mixing on the space of configurations, and hence can be used for finding a  random configuration nearly uniformly. 
We have two points to make here.  

\medskip\noindent
First, uniform sampling of the configuration model can be achieved in a much more straightforward manner than a Markov chain. One simply needs to provide random permutations to describe the matchings for each complete bipartite graph. Generating uniformly random permutations is both simpler and much faster than implementing the proposed Markov chain.

\medskip\noindent
Second, and more importantly, uniform sampling from the space of configurations does {\em not} yield uniform sampling from the space of all multigraphs. To see this, consider the JDM $\J = \left[\begin{smallmatrix} 0& 0\\ 0& 3 \end{smallmatrix}\right]$. It is easy to verify that this is the JDM for precisely 5 multigraphs on 3 labeled vertices. The first is a $C_3$, the second is three vertices with one loop each, and the other 3 graphs each consist of one loop and a double edge. In the configuration model, there is only one bipartite graph, consisting of 6 mini-vertices and 6 edge-vertices. Hence there are $6! =720$ different configurations possible. A computation reveals that 384 correspond to $C_3$, 48 correspond to the graph with three loops, and 96 correspond to each of the loop + double-edge graphs. In this case, a uniform distribution on configurations yields a distribution that favors the $C_3$ over the three-loop graph by a factor of 8, certainly far from uniform sampling. 

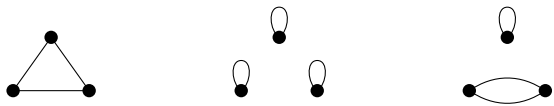
\begin{figure}[h!]
\begin{center}
\begin{tikzpicture}[line cap=round,line join=round,>=triangle 45]

\fill [color=black] (0,0) circle (2.5pt);
\fill [color=black] (1/2, .70711 ) circle (2.5pt);
\fill [color=black] (1,0) circle (2.5pt);
\draw (0,0)--(1/2,.70711)--(1,0)--(0,0);

\fill [color=black] (0,0)+(3,0) circle (2.5pt);
\draw (3,0) to [out=145, in=180] (3,.4) to[out=0, in=35] (3,0);
\fill [color=black] (1/2, .70711 )+(3,0) circle (2.5pt);
\draw (3.5, .70711 ) to [out=145, in=180] (3+1/2, .70711 +.4) to[out=0, in=35] (3+1/2, .70711 );
\fill [color=black] (1,0) +(3,0)circle (2.5pt);
\draw (4,0) to [out=145, in=180] (4,.4) to[out=0, in=35] (4,0);

\fill [color=black] (0,0) +(6,0) circle (2.5pt);
\draw (6,0) to [out=-35, in=215] (7,0);
\fill [color=black] (1/2, .70711 ) +(6,0)circle (2.5pt);
\draw (6.5, .70711 ) to [out=145, in=180] (6+1/2, .70711 +.4) to[out=0, in=35] (6+1/2, .70711 );
\fill [color=black] (1,0)+(6,0) circle (2.5pt);
\draw (6,0) to [out=35, in=180-35] (7,0);
\end{tikzpicture}
\end{center}
\newsavebox{\smlmat}
\savebox{\smlmat}{$\left[\begin{smallmatrix}0&0\\0&3\end{smallmatrix}\right]$}
\caption{The three non-isomorphic multigraphs realizing the JDM $\mathcal{J} =$~\usebox{\smlmat}.The first two have unique labelings, while the third has three distinct labelings.}
\end{figure}

\subsection*{Sampling from chain $\cB$}
\noindent
Stanton and Pinar noted (and we agree)  that bounding the mixing time of the Markov chain $\cB$ seems to be a very difficult problem without some new ideas or techniques. They conducted a series of experiments, running chain $\cB$ multiple times on 11 sample graphs of varying sizes, and measured the {\em autocorrelation time} of the chain, a measurement that they say can be substituted for the mixing time. 

\medskip\noindent
Indeed, for a particular sampling from a fixed JDM, autocorrelation may be an excellent metric for determining how many steps to take between samples, and a way to check that samples  are close to uniformly chosen. Also, the autocorrelation experiments do provide some evidence for rapid mixing. We take issue with their claim that this evidence shows that the Markov chain is fast-mixing in practice. 

\medskip\noindent
Our view is the following: autocorrelation is an excellent tool to show that a Markov chain is {\em not} fast mixing. However, it is not powerful enough to show the opposite. It can show that one particular run of the Markov chain is good, but it cannot predict that all outcomes will be good as well. Even if an experiment is repeated many times, it may give some confidence that the chain is fast-mixing for one the particular JDM, but this may not relate to how it may act on another.

\section{Further Directions}\label{sec:further} 
\medskip\noindent
Because the space of realizations of a joint degree matrix is connected, one could use the Markov chain $\cB$ above to pick a random realization. Proving the rapid mixing of this chain, though seemingly an intractable problem at the present, would be a clear step forward. 

\medskip \noindent
Another option would be to {\em generate} a random realization. In \cite{StarConst}, the authors develop a constrained version of the Havel-Hakimi algorithm for realizing a degree sequence. This constrained version is able to directly generate every possible realization. Furthermore, in \cite{StarConstSampling}, the authors determine a way to provide a weight corresponding to each degree sequence realization, which can be used to make the sampling uniform. It would be of both theoretical and practical interest to do the same for joint degree matrices. 

\medskip \noindent
Finally, it is interesting to note that for degree sequences, there are at least two distinct descriptions of when a sequence is graphical: One by Erd\H{o}s-Gallai \cite{EG}, and the other by Havel \cite{havel}. The characterization by Havel lends itself to a simple algorithmic implementation for building realizations, much moreso than that of Erd\H{o}s and Gallai.  Perhaps there is also a second description of the matrices that can be realized as the joint degree matrix of a graph, one which can be used directly construct many different graphical realizations. 

\bibliographystyle{plain}

\end{document}